\documentclass[11pt]{article}

\usepackage{amsmath,amssymb,graphicx}
\usepackage{hyperref}

\usepackage{multicol,color,longtable}
\usepackage[usenames,dvipsnames,svgnames,table]{xcolor}
\usepackage[prependcaption,bordercolor=white,backgroundcolor=gray!30,linecolor=black,colorinlistoftodos,textsize=tiny]{todonotes}
\reversemarginpar
\usepackage{tabu}
\usepackage{cleveref}
\usepackage{xargs}
\usepackage{mathdots}
\usepackage{bbm}
\usepackage{bm}

\usepackage{amsthm}

\newtheorem{lemma}{Lemma}
\newtheorem{thm}{Theorem}
\newtheorem{corollary}{Corollary}
\theoremstyle{definition}
\newtheorem{defn}{Definition}
\theoremstyle{remark}
\newtheorem*{remark}{Remark}

\newcommandx{\TODO}[2][1=]{\todo[fancyline,linecolor=red,backgroundcolor=red!25,bordercolor=red,#1]{#2}}
\newcommandx{\quest}[2][1=]{\todo[fancyline,linecolor=blue,backgroundcolor=blue!25,bordercolor=blue,#1]{#2}}
\newcommandx{\missref}[1][1=]{\todo[fancyline,linecolor=OliveGreen,backgroundcolor=OliveGreen!25,bordercolor=OliveGreen,#1]{Reference}}
\newcommandx{\improvement}[2][1=]{\todo[fancyline,linecolor=Plum,backgroundcolor=Plum!25,bordercolor=Plum,#1]{#2}}
\newcommandx{\thiswillnotshow}[2][1=]{\todo[disable,#1]{#2}}

\definecolor{darkred}{rgb}{0.65,0.15,0}
\hypersetup{pdfborder={0 0 0},colorlinks=true,urlcolor=darkred,citecolor=blue,linkcolor=darkred,linktocpage=true}

\allowdisplaybreaks

\newcommand{\nats}{\mathbb{N}}
\newcommand{\ints}{\mathbb{Z}}
\newcommand{\reals}{\mathbb{R}}

\newcommand{\rats}{\mathbb{Q}}

\newcommand{\adeles}{\mathbb{A}}

\newcommand{\dash}{\ensuremath{\text{---}}}

\newcommand{\GL}{\operatorname{GL}}
\newcommand{\SL}{\operatorname{SL}}
\newcommand{\SO}{\operatorname{SO}}

\newcommand{\T}{\operatorname{T}}

\newcommand{\qtextq}[1]{{\quad \text{#1} \quad}}
\newcommand{\qtext}[1]{{\quad \text{#1}}}

\makeatletter

\@addtoreset{equation}{section}
\makeatother

\begin{document}

\thispagestyle{empty}

\mbox{ }
\vspace{40mm}

\begin{center}
  {\LARGE \bf Global Iwasawa-decomposition of $\SL ( n,\adeles_{\rats} )$ }\\[10mm]

  \vspace{8mm}
  \normalsize
  {\large  Olof Ahl\'en}

  \vspace{10mm}
  {\it Max Planck Institute for Gravitational Physics (Albert Einstein Institute)\\
  Am M\"{u}hlenberg 1, DE-14476 Potsdam, Germany}

  \vspace{15mm}

  \hrule

  \vspace{10mm}

  \begin{tabular}{p{12cm}}
    {\small
      We discuss the Iwasawa-decomposition of a general matrix in $\SL( n, \rats_p )$ and $\SL( n, \reals )$. For $\SL( n, \rats_p )$ we define an algorithm for computing a complete Iwasawa-decomposition and give a formula parameterizing the full family of decompositions. Furthermore, we prove that the $p$-adic norms of the coordinates on the Cartan torus are unique across all decompositions and give a closed formula for them which is proven using induction. For the case $\SL( n, \reals )$, the decomposition is unique and we give formulae for the complete decomposition which are also proven inductively. Lastly we outline a method for deriving the norms of the coordinates on the Cartan torus in the framework of representation theory. This yields a simple formula valid globally which expresses these  norms in terms of the vector norms of generalized Pl\"ucker coordinates.
    }
  \end{tabular}

  \vspace{10mm}
  \hrule

\end{center}
\vfill

\newpage

\setcounter{page}{1}

\tableofcontents


\section{Introduction}
Let $G(F)$ be a connected semisimple Lie group over a field $F$ with Lie algebra $\mathfrak{g}(F)$. The Iwasawa decomposition states that $G(F)$ may be written\footnote{In the literature, one additionally encounters the equivalent decomposition $G = KAN$. Which one is used is a matter of convention, and statements and proofs within this paper easily carry over.}
\begin{equation}
  G(F) = N(F) A(F) K(F)
  .
\end{equation}
where $N(F)$ is a nilpotent subgroup generated by the positive Chevalley generators in $\mathfrak{g}(F)$, $A(F)$ is an abelian subgroup generated by a maximal abelian subalgebra of $\mathfrak{g}(F)$ and $K(F)$ is a maximal compact subgroup of $G(F)$. The group elements of $N(F)$, $A(F)$ and $K(F)$ in the decomposition of some group element in $g \in G(F)$ are called the nilpotent-, semisimple- and compact parts of $g$ respectively.

This paper discusses the field of real numbers and the field of $p$-adic numbers. We recall some terminology regarding $p$-adic numbers and their norm.

\begin{defn}[$p$-adic numbers]
  Let $p$ be a prime number. The $\boldsymbol{p}$\textbf{-adic norm} of a rational number $\frac{m'}{n'} p^a$, where the primes denote that the integers $m$ and $n$ carry no factors of $p$, is defined as
  \begin{equation}
    \left| \frac{m'}{n'} p^a \right|_p = p^{-a}
    .
  \end{equation}

  The $\boldsymbol{p}$\textbf{-adic numbers} $\rats_p$ are the completion of the rational numbers with respect to the $p$-adic norm.

  The $\boldsymbol{p}$\textbf{-adic integers} $\ints_p$ are the unit ball in $\rats_p$,
  \begin{equation}
    \ints_p \equiv \left\{ x \in \rats_p : |x|_p \leq 1 \right\}
    .
  \end{equation}

  The $\boldsymbol{p}$\textbf{-adic units} $\ints_p^\times$ are the unit circle in $\rats_p$,
  \begin{equation}
    \ints_p^{\times} \equiv \left\{ x \in \rats_p : |x|_p = 1 \right\}
    .
  \end{equation}
\end{defn}
\begin{remark}
  The $p$-adic norm satisfies the ultrametric property
  \begin{equation}
    |x+y|_p \leq \max \{|x|_p, |y|_p\} \qtext{where } x \text{ and } y \in \rats_p
    .
  \end{equation}
  This property renders the $p$-adic numbers non-archimedean.
\end{remark}
\begin{remark}
  The $p$-adic integers form a ring and sit compactly inside the $p$-adic numbers.
\end{remark}

In the case $F = \reals$, the maximal compact subgroup is the exponentiation of the subalgebra of $\mathfrak{g}(F)$ consisting of the fixed point elements of the Chevalley involution.

By contrast, in the non-archimedean case $F = \rats_p$, the notion of a maximal compact subgroup is defined by virtue of the fact that $\ints_p$ forms a compact ring inside $\rats_p$. For matrix groups (which are linear algebraic subgroups of $\GL(n, \rats_p)$), the maximal compact subgroup is defined by restricting $G( \rats_p )$ to the subgroup of integers points
\begin{equation}
  K( \ints_p ) = G( \rats_p ) \cap \GL( n, \ints_p ) \equiv G(\ints_p)
  .
\end{equation}
$K( \ints_p )$ defined in this way then sits compactly inside $G(\rats_p)$.

The Iwasawa-decomposition is relevant for maximal parabolic Eisenstein-series, which are automorphic forms and constructed by averaging a character of a Lie group over a discrete subgroup. Since a character is only sensitive to the abelian part of a group, it is useful to possess a formula which rewrites a group element in Iwasawa form, or at least that can extract the semisimple part $A$ given an arbitrary group element. Furthermore when dealing with automorphic forms, it is oftentimes advantageous to operate over the ring of adeles $\adeles_\rats$ which calls for the formula in question to be applicable both for groups over real numbers and groups over $p$-adic numbers. For a review of working with Eisenstein series over the adeles, see \cite{fleig2015eisenstein}.

This paper treats $G(F) = \SL (n, F )$ in the fundamental representation for $F = \reals$ and $F = \rats_p$ separately.
In the $\SL (n, F )$-case, the group element $N$ becomes a unit upper triangular matrix and $A$ becomes a diagonal matrix of unit determinant. We have the following definition for their matrix elements:
\begin{defn}[Axions, dilatons and Cartan torus]
  For a matrix in $\SL (n, F )$ written in Iwasawa form, the matrix elements above the diagonal in $N$ are called \textbf{axions}. The matrix $A$ is called the \textbf{Cartan torus} and when parameterized as
  \begin{equation}
    \label{eqn:dilatons}
    A =
    \left(
    \begin{array}{ccccc}
      y_1 \\
      & \frac{y_2}{y_1} \\
      & & \ddots \\
      & & & \frac{y_{n-1}}{y_{n-2}} \\
      & & & & \frac{1}{y_{n-1}} \\
    \end{array}
    \right)
    ,
  \end{equation}
  the $y$'s are called \textbf{dilatons}.
\end{defn}
\begin{remark}
  It can be desirable to define $y_0 \equiv y_n \equiv 1$ and write $\frac{y_1}{y_0}$ and $\frac{y_n}{y_{n-1}}$ for the first and last elements respectively.
\end{remark}
\begin{remark}
  The vocabulary comes from String Theory, where the moduli are referred to as axions and dilatons.
\end{remark}<++>
In the $p$-adic case, the Iwasawa-decomposition is not unique but the norms of the dilatons are unique, which we prove below. It is these norms that are important for Eisenstein series and a formula to compute them given an arbitrary $p$-adic matrix is derived. This formula is the main result of this paper.

In the real case, the Iwasawa-decomposition of a given group element is unique and a formula is given for the complete decomposition. 

Throughout this paper, we use the convention that two integers separated by ellipses or a dash are to be interpreted as an integer spaced interval. For example, given two integers $a$ and $b$, writing
\begin{equation*}
  \begin{aligned}
    &a, \dots, b \qtextq{is shorthand for} a, a+1, a+2, \dots, b-2, b-1, b \qtext{and} \\
    &\epsilon_{i_a \dash i_b} \qtextq{is shorthand for} \epsilon_{i_{a} i_{a+1} i_{a+2} \dots i_{b-2} i_{b-1} i_{b}}
    .
  \end{aligned}
\end{equation*}
If $a>b$, the interval is defined to be empty. Furthermore, we denote matrix elements $M_{ij}$ of a matrix $M$ by a pair of indices, denoting the row and column respectively.

My personal motivation to pursue the results presented here comes from the study of automorphic forms that appear in toroidal compactifications of type IIB String Theory \cite{green1997effects,green2006duality,green2010eisenstein,green2015small}, where in the case of compactification on a 3-torus, one finds Eisenstein series on the group $\SL(5, \reals)$. The work \cite{fleig2015eisenstein} provides an introduction for mathematicians to this area of physics. 

\section{Decomposing $\SL( n, \rats_p )$}
We begin with a more precise statement of Iwasawa-decomposition for matrices in $\SL( n, \rats_p )$.
\begin{thm}[Iwasawa-decomposition for $\SL( n, \rats_p )$]
  A matrix $M \in \SL( n, \rats_p )$ may be written
  \begin{equation}
    M = NAK
  \end{equation}
  where $N \in \SL( n, \rats_p )$ is unit upper triangular, $A \in \SL( n, \rats_p )$ is diagonal and $K \in \SL( n, \ints_p )$.
  \label{thm:padiciwasawa}
\end{thm}
\begin{remark}
  By counting degrees of freedom, we find $n^2-1$ on the left hand side and $\frac{n(n-1)}{2}+n-1+n^2-1$ on the right hand side. The decomposition is hence not unique and we expect a $\left( \frac{n(n+1)}{2}-1 \right)$-family of decompositions.
\end{remark}
\subsection{Preliminaries}
We proceed by giving some definitions that will be relevant in the pursuit of a formula for $p$-adic Iwasawa-decomposition.
\begin{defn}[Minor]
  Given an $m\times n$ matrix $M$, a \textbf{minor of order} $\bm{k}$
  \begin{equation}
    M\left(
    \begin{smallmatrix}
      r_{1} & \dots & r_{k} \\
      c_{1} & \dots & c_{k} \\
    \end{smallmatrix}
    \right)
    \label{eqn:minor}
  \end{equation}
  is the determinant of the submatrix of $M$ obtained by only picking the $k$ rows $\{r_{i}\}$ and $k$ columns $\{c_{i}\}$.

  If the rows and columns agree, i.e.\ $r_i = c_i$ for all $i\in \{1, \dots, k\}$, then the minor is called a \textbf{principal minor}.

  If the rows selected are the first $k$ rows in order,
  \begin{equation}
    M\left(
    \begin{smallmatrix}
      1 & \dots & k \\
      c_{1} & \dots & c_{k} \\
    \end{smallmatrix}
    \right)
    ,
  \end{equation}
  the minor is called a \textbf{leading minor} while if they are the last $k$ rows in order,
  \begin{equation}
    M\left(
    \begin{smallmatrix}
      m-k+1 & \dots & m \\
      c_{1} & \dots & c_{k} \\
    \end{smallmatrix}
    \right)
    ,
  \end{equation}
  we will call it an \textbf{anti-leading minor}.

  Hence, the minors
  \begin{equation}
    M\left(
    \begin{smallmatrix}
      1 & \dots & k \\
      1 & \dots & k \\
    \end{smallmatrix}
    \right)
    \quad{}\text{and}\quad{}
    M\left(
    \begin{smallmatrix}
      m-k+1 & \dots & m \\
      n-k+1 & \dots & n \\
    \end{smallmatrix}
    \right)
  \end{equation}
  are called the \textbf{leading principal minor} and the \textbf{anti-leading principal minor} of order $k$ respectively.

  The empty minor is defined as
  \begin{equation}
    M\left(
    \begin{smallmatrix}
      {} \\
      {} \\
    \end{smallmatrix}
    \right)
    \equiv 1.
  \end{equation}
\end{defn}
\begin{remark}
  A minor is totally antisymmetric under permutations of the rows as well as the columns. Hence it vanishes if some $r$'s coincide or some $c$'s coincide.
\end{remark}
\begin{remark}
  A minor can be expanded along a row or column according the formula
  \begin{equation}
    \begin{aligned}
      M\left(
      \begin{smallmatrix}
        r_{1} & \dots & r_{k} \\
        c_{1} & \dots & c_{k} \\
      \end{smallmatrix}
      \right)
      ={}&
      \sum_{a=1}^{k}
      (-1)^{a+1}
      M\left(
      \begin{smallmatrix}
        r_{1} \\
        c_{a} \\
      \end{smallmatrix}
      \right)
      M\left(
      \begin{smallmatrix}
        r_{2} & \dots & r_{a} & r_{a+1} & \dots & r_k \\
        c_{1} & \dots & c_{a-1} & c_{a+1} & \dots & c_k \\
      \end{smallmatrix}
      \right)
      = \\
      ={}&
      \sum_{a=1}^{k}
      (-1)^{a+1}
      M\left(
      \begin{smallmatrix}
        r_{a} \\
        c_{1} \\
      \end{smallmatrix}
      \right)
      M\left(
      \begin{smallmatrix}
        r_{1} & \dots & r_{a-1} & r_{a+1} & \dots & r_k \\
        c_{2} & \dots & c_{a} & c_{a+1} & \dots & c_k \\
      \end{smallmatrix}
      \right)
    \end{aligned}
  \end{equation}
  called Laplace expansion.
\end{remark}

Next we recall the LU-decomposition of a matrix and an accompanying lemma.
\begin{lemma}
  \label{lemma:plm}
  Given a non-singular square matrix $M$ of size $n$, there is a permutation matrix $P$ such that the leading principal minors of $MP$ are all non-zero.
\end{lemma}
A proof of this can be found in \cite{horn2012matrix} and below we prove a more powerful version of this lemma adapted for anti-leading minors.

Given a non-singular square matrix $M$, one can then always find a permutation matrix $P$ such that LU-decomposition of $MP$ is possible.
\begin{thm}[LU-decomposition]
  \label{thm:lu}
Let $F$ be a field. A matrix $M \in \SL (n, F )$ can be written as
  \begin{equation}
    \begin{aligned}
      M ={}& MP P^{-1} = LDUP^{-1} = \\
      ={}&
      \left(
      \begin{array}{cccc}
        1 & & & \\
        l_{21} & 1 & & \\
        \vdots & \vdots & \ddots & \\
        l_{n1} & l_{n2} & \dots & 1
      \end{array}
      \right)
      \left(
      \begin{array}{cccc}
        y_1 &         &        & \\
        & \frac{y_2}{y_1} &        & \\
        &         & \ddots & \\
        &         &        & \frac{1}{y_{n-1}} \\
      \end{array}
      \right)
      \left(
      \begin{array}{cccc}
        1 & u_{12} & \dots & u_{n1} \\
        & 1      & \dots & u_{2n} \\
        &        & \ddots& \vdots \\
        &        &       & 1 \\
      \end{array}
      \right)
      P^{-1}
    \end{aligned}
  \end{equation}
  where $D$ has unit determinant and its elements are given by
  \begin{equation}
    y_{p} = 
    (MP)\left(
    \begin{smallmatrix}
      1 & \dots & p \\
      1 & \dots & p \\
    \end{smallmatrix}
    \right)
    ,
  \end{equation}
  $L$ is unit lower triangular with matrix elements
  \begin{equation}
    l_{ip} = 
    \left.
    (MP)\left(
    \begin{smallmatrix}
      1 & \dots & p-1 & i \\
      1 & \dots & p-1 & p \\
    \end{smallmatrix}
    \right)
    \right/
    (MP)\left(
    \begin{smallmatrix}
      1 & \dots & p \\
      1 & \dots & p \\
    \end{smallmatrix}
    \right)
    =
    \left.
    (MP)\left(
    \begin{smallmatrix}
      1 & \dots & p-1 & i \\
      1 & \dots & p-1 & p \\
    \end{smallmatrix}
    \right)
    \right/
    y_p
    \quad{}\text{$i\geq p$,}
  \end{equation}
  $U$ is unit upper triangular with matrix elements
  \begin{equation}
    u_{pi} = 
    \left.
    (MP)\left(
    \begin{smallmatrix}
      1 & \dots & p-1 & p \\
      1 & \dots & p-1 & i \\
    \end{smallmatrix}
    \right)
    \right/
    (MP)\left(
    \begin{smallmatrix}
      1 & \dots & p \\
      1 & \dots & p \\
    \end{smallmatrix}
    \right)
    =
    \left.
    (MP)\left(
    \begin{smallmatrix}
      1 & \dots & p-1 & p \\
      1 & \dots & p-1 & i \\
    \end{smallmatrix}
    \right)
    \right/
    y_p
    \quad{}\text{$p\leq i$}
  \end{equation}
  and $P$ is almost\footnote{\label{footnote:minussigns}In the case that $P$ is an odd permutation, we may replace it with for example
    $
    \left(
    \begin{smallmatrix}
      -1 \\
       & 1 \\
       & & \ddots \\
       & & & 1
    \end{smallmatrix}
    \right)
    P
    $
    to give it a positive determinant and thereby preserve the determinant condition on both sides.}a permutation matrix.
\end{thm}
A proof of the formulae above can be found in \cite{householder2013theory}.

\subsection{New results}
For the purposes of Iwasawa-decomposition, we will make use of a UL-decomposition rather than the conventional LU-decomposition. In this case, leading principal minors get replaced by anti-leading principal minors.\footnote{Switching between UL and LU introduces no complications for proofs, which work analogously for the two. Which one is needed depends on which of the conventions $KAN$ or $NAK$ is used.} We begin by proving a more powerful version of lemma \ref{lemma:plm} which lets us construct a UL-decomposition with the added property that one may choose the rightmost column in the matrix to be decomposed freely, as long as its bottom element is nonzero. This stronger UL-decomposition will turn out to be useful in computing the Iwasawa-decomposition of an arbitrary matrix in $\SL( n, \rats_p )$.
\begin{lemma}
  Given a non-singular square matrix $M$ of size $n$, there is a permutation matrix $\Pi_a$ such that the anti-leading principal minors of $M\Pi_a$ are all non-zero where $\Pi_a$ moves column $a$ to the rightmost position and we require the bottom element of column $a$ to be non-zero.
\end{lemma}
\begin{proof}
  The permutation of columns is realized by multiplication of a permutation matrix from the right. The proof works by induction and the cases $n=1$ and $n=2$ are obvious. Assume that the statement holds true for matrices of size up to and including $n-1$ and consider a matrix of size $n$. Start by permuting the columns of $M$ so that column $a$ is not in the leftmost position and consider the $(n-1)\times n$ matrix obtained by deleting the top row of the permuted matrix. This matrix contains $n-1$ linearly independent rows and hence also contains $n-1$ linearly independent columns where column $a$ can be assumed to be one of them, since it is not a zero-column. Permute these columns to the rightmost $n-1$ positions and apply the induction hypothesis to the bottom right $(n-1)\times(n-1)$ block of the permuted matrix, which is allowed since this block is guaranteed to be a non-singular matrix. This establishes that the first $n-1$ anti-leading principal minors of $M\Pi$ are non-zero and the rightmost column of $M \Pi$ is column $a$ of $M$, where $\Pi$ denotes the resulting permutation matrix. The last anti-leading principal minor is simply $\det(M \Pi)$ which is non-zero since $M \Pi$ is non-singular. Peano's axiom of induction now establishes the lemma.
\end{proof}

\begin{thm}[Strong UL-decomposition]
Let $F$ be a field. A matrix $M \in \SL (n, F )$ with $M_{na} \neq 0$ can be written as
  \begin{equation}
    \begin{aligned}
      M ={}& M \Pi_a \Pi_a^{-1} = V \Delta \Lambda \Pi_a^{-1} = \\
      ={}&
      \left(
      \begin{array}{cccc}
        1 & v_{12} & \dots & v_{n1} \\
        & 1      & \dots & v_{2n} \\
        &        & \ddots& \vdots \\
        &        &       & 1 \\
      \end{array}
      \right)
      \left(
      \begin{array}{cccc}
        \eta_1 &         &        & \\
        & \frac{\eta_2}{\eta_1} &        & \\
        &         & \ddots & \\
        &         &        & \frac{1}{\eta_{n-1}} \\
      \end{array}
      \right)
      \left(
      \begin{array}{cccc}
        1 & & & \\
        \lambda_{21} & 1 & & \\
        \vdots & \vdots & \ddots & \\
        \lambda_{n1} & \lambda_{n2} & \dots & 1
      \end{array}
      \right)
      \Pi_a^{-1}
    \end{aligned}
  \end{equation}
  where $\Delta$ has unit determinant and its elements are given by
  \begin{equation}
    \label{eqn:etas}
    \eta_{p} = 
    \left(
    \left(M \Pi_a\right)\left(
    \begin{smallmatrix}
      p+1 & \dots & n \\
      p+1 & \dots & n \\
    \end{smallmatrix}
    \right)
    \right)^{-1}
    ,
  \end{equation}
  $V$ is unit upper triangular with matrix elements
  \begin{equation}
    \label{eqn:vs}
    v_{pi} = 
    \left.
    \left(M \Pi_a\right)\left(
    \begin{smallmatrix}
      p & i+1 & \dots & n \\
      i & i+1 & \dots & n \\
    \end{smallmatrix}
    \right)
    \right/
    \left(M \Pi_a\right)\left(
    \begin{smallmatrix}
      i & \dots & n \\
      i & \dots & n \\
    \end{smallmatrix}
    \right)
    =
    \left(M \Pi_a\right)\left(
    \begin{smallmatrix}
      p & i+1 & \dots & n \\
      i & i+1 & \dots & n \\
    \end{smallmatrix}
    \right)
    \eta_{i-1}
    \quad{}\text{$p\leq i$,}
  \end{equation}
  $\Lambda$ is unit lower triangular with matrix elements
  \begin{equation}
    \label{eqn:lambdas}
    \lambda_{ip} = 
    \left.
    \left(M \Pi_a\right)\left(
    \begin{smallmatrix}
      i & i+1 & \dots & n \\
      p & i+1 & \dots & n \\
    \end{smallmatrix}
    \right)
    \right/
    \left(M \Pi_a\right)\left(
    \begin{smallmatrix}
      i & \dots & n \\
      i & \dots & n \\
    \end{smallmatrix}
    \right)
    =
    \left(M \Pi_a\right)\left(
    \begin{smallmatrix}
      i & i+1 & \dots & n \\
      p & i+1 & \dots & n \\
    \end{smallmatrix}
    \right)
    \eta_{i-1}
    \quad{}\text{$i\geq p$}
  \end{equation}
  and $\Pi_a$ is a permutation matrix which moves column $a$ to the rightmost position, subject to the caveat explained in footnote \ref{footnote:minussigns}.
\end{thm}
\begin{proof}
  Denote
  \begin{align}
    W = 
    \left(
    \begin{array}{ccc}
      & & 1 \\
      & \iddots & \\
      1 & & \\
    \end{array}
    \right)
    =
    W^{-1}
    .
  \end{align}
  Note that conjugating a matrix by $W$ ``rotates'' it by $180^\circ$. Write the matrix $WM \Pi_a W^{-1}$ using the LU-decomposition (here we get $P=\mathbbm{1}$ thanks to the action of $\Pi_a$)
  \begin{equation}
    WM \Pi_a W^{-1} = LDU
    .
  \end{equation}
  Solve for $M$ and write it as
  \begin{equation}
    M = \underbrace{W^{-1}LW}_V \underbrace{W^{-1}DW}_\Delta \underbrace{W^{-1}UW}_\Lambda \Pi_a^{-1}
    .
  \end{equation}
  The formula for the matrix elements of $V$, $\Delta$ and $\Lambda$ follow from theorem \ref{thm:lu}. Alternatively, the formulae may be proven from first principles using the same technique as for the UL-decomposition in \cite{householder2013theory}.
\end{proof}
\begin{remark}
  All minors in the formulae above contain column $a$ of the original matrix $M$.
\end{remark}

Next, we show how the family of Iwasawa-decompositions of a given matrix may be enumerated given any Iwasawa-decomposition of the matrix. A corollary of this is that the norms of the dilatons are unique across all Iwasawa-decompositions.
\begin{thm}
  \label{thm:nonuniqueness}
  Given an Iwasawa decomposition $NAK$ of $M$, all other Iwasawa-decompositions $N'A'K'$ can be found by writing $N' = NAXA^{-1}$, $A' = AY$ and $K' = (XY)^{-1}K$ and letting $X$ range over all unit upper triangular matrices in $\SL( n, \ints_p )$ and $Y$ over all diagonal matrices in $\SL( n, \ints_p^\times )$.
\end{thm}
\begin{proof}
  Given two Iwasawa-decompositions of $M$
  \begin{equation}
    M = N_1 A_1 K_1 = N_2 A_2 K_2
    ,
  \end{equation}
  define the matrix
  \begin{equation}
    Z = \left( N_1 A_1 \right)^{-1} N_2 A_2 = K_1 K_2^{-1}
    .
  \end{equation}
  Note that $Z \in \SL( n, \ints_p )$ is upper triangular and $\left| \det Z \right|_p = |1|_p = 1$ is the product of its diagonal elements. This implies that the diagonal elements must all be $p$-adic units, since if some had a norm less than one, others must have a norm greater than one which would violate $Z \in \SL( n, \ints_p )$. We can therefore decompose $Z = XY$ with $X$ and $Y$ as above. The matrix $Z$ takes the first decomposition into the second by writing
  \begin{equation}
    N_1 A_1 K_1 = N_1 A_1 Z Z^{-1} K_1 = N_1 A_1 \left( N_1 A_1 \right)^{-1} N_2 A_2 K_2 K_1^{-1} K_1 = N_2 A_2 K_2
    .
  \end{equation}
  We have at the same time
  \begin{equation}
    N_1 A_1 K_1  = N_1  A_1  XY (XY)^{-1} K_1  = \underbrace{N_1  A_1  X A_1 ^{-1}}_{N_2} \underbrace{A_1  Y}_{A_2} \underbrace{(XY)^{-1} K_1}_{K_2}
    .
  \end{equation}
  From the form of $X$ and $Y$, it is clear that $N_2$ and $Y_2$ have the appropriate form for Iwasawa-decomposition. The same is true for $K_2$ since $(XY)^{-1} \in \SL( n, \ints_p )$.

  This proves that any two Iwasawa-decompositions are related in the way claimed by the theorem and by varying $X$ and $Y$, one generates all other Iwasawa-decompositions.
\end{proof}
\begin{remark}
  $X$ and $Y$ parameterize the $\left( \frac{n(n+1)}{2}-1 \right)$-family of decompositions mentioned in the remark after theorem \ref{thm:padiciwasawa}.
\end{remark}
\begin{corollary}
  For a matrix $M \in \SL( n, \rats_p )$, the norms of the dilatons are unique across all Iwasawa-decompositions of $M$.
\end{corollary}
\begin{proof}
  By theorem \ref{thm:nonuniqueness}, the semisimple part of any Iwasawa decomposition of $M$ is of the form $AY$ where $A$ is the semisimple part of some Iwasawa-decomposition and $Y \in \SL( n, \ints_p^\times )$ is diagonal, hence the norms of the diagonal elements are unchanged. Regarding the dilatons in \eqref{eqn:dilatons}, we get that the norm of $y_1$ in unchanged, and hence also that of $y_2$ etcetera.
\end{proof}

Next we prove two lemmas regarding minors which will be useful later.

\begin{lemma}
  An $m\times n$ matrix $M$ obeys
  \begin{equation}
    \begin{aligned}
      M\left(
      \begin{smallmatrix}
        r_{1} & \dots & r_{k} \\
        c_{1} & \dots & c_{k} \\
      \end{smallmatrix}
      \right)
      M\left(
      \begin{smallmatrix}
        r_{2} & \dots & r_{k} \\
        d_{2} & \dots & d_{k} \\
      \end{smallmatrix}
      \right)
      = \sum_{a=1}^k (-1)^{a+1}
      M\left(
      \begin{smallmatrix}
        r_{1} & r_{2} & \dots & r_{k} \\
        c_{a} & d_{2} & \dots & d_{k} \\
      \end{smallmatrix}
      \right)
      M\left(
      \begin{smallmatrix}
        r_{2} & \dots & r_{a} & r_{a+1} & \dots & r_{k} \\
        c_{1} & \dots & c_{a-1} & c_{a+1} & \dots & c_{k} \\
      \end{smallmatrix}
      \right)
    \end{aligned}
  \end{equation}
  where $r_i \in \{1, \dots, m\}$ and $c_i \in \{1, \dots, n\}$ and $d_i \in \{1, \dots, n\}$ and $k \in \nats$.
  \label{lemma:1}
\end{lemma}
\begin{proof}
  For readability, we will drop $r$ and write $r_{i}$ just as $i$. To prove the identity, we expand both sides using Laplace expansion. Start by expanding the first factor of the left hand side as
  \begin{align}
    \begin{aligned}
      \mathrm{LHS} ={}&
      M\left(
      \begin{smallmatrix}
        1 & \dots & k & \\
        c_{1} & \dots & c_{k} \\
      \end{smallmatrix}
      \right)
      M\left(
      \begin{smallmatrix}
        2 & \dots & k \\
        d_{2} & \dots & d_{k} \\
      \end{smallmatrix}
      \right)
      = \\
      ={}&
      \sum_{a=1}^k (-1)^{a+1}
      M\left(
      \begin{smallmatrix}
        1 \\
        c_{a} \\
      \end{smallmatrix}
      \right)
      M\left(
      \begin{smallmatrix}
        2 & \dots & a & a+1 & \dots & k \\
        c_{1} & \dots & c_{a-1} & c_{a+1} & \dots & c_k \\
      \end{smallmatrix}
      \right)
      M\left(
      \begin{smallmatrix}
        2 & \dots & k \\
        d_{2} & \dots & d_{k} \\
      \end{smallmatrix}
      \right)
      .
    \end{aligned}
  \end{align}
  Expand the first factor in the sum of the right hand side as
  \begin{equation}
    \begin{aligned}
      \mathrm{RHS} &={}
      \sum_{a=1}^k (-1)^{a+1}
      M\left(
      \begin{smallmatrix}
        1 & 2 & \dots & k \\
        c_{a} & d_{2} & \dots & d_{k} \\
      \end{smallmatrix}
      \right)
      M\left(
      \begin{smallmatrix}
        2 & \dots & a & a+1 & \dots & k \\
        c_{1} & \dots & c_{a-1} & c_{a+1} & \dots & c_{k} \\
      \end{smallmatrix}
      \right)
      = \\
      &
      \begin{aligned}
        ={}
        \sum_{a=1}^k (-1)^{a+1}
        &
        \Bigg(
        \underbrace{
          M\left(
          \begin{smallmatrix}
            1 \\
            c_{a} \\
          \end{smallmatrix}
          \right)
          M\left(
          \begin{smallmatrix}
            2 & \dots & k \\
            d_{2} & \dots & d_{k} \\
          \end{smallmatrix}
          \right)
        }_{\mathrm{I}}
        +
        \underbrace{
          \sum_{b=2}^{k} (-1)^{b+1}
          M\left(
          \begin{smallmatrix}
            b \\
            c_{a} \\
          \end{smallmatrix}
          \right)
          M\left(
          \begin{smallmatrix}
            1 & \dots & b-1 & b+1 & \dots & k \\
            d_{2} & \dots & d_{b} & d_{b+1} & \dots & d_{k} \\
          \end{smallmatrix}
          \right)
        }_{\mathrm{II}}
        \Bigg)
        \\
        &
        M\left(
        \begin{smallmatrix}
          2 & \dots & a & a+1 & \dots & k \\
          c_{1} & \dots & c_{a-1} & c_{a+1} & \dots & c_{k} \\
        \end{smallmatrix}
        \right)
        .
      \end{aligned}
    \end{aligned}
  \end{equation}
  The term labelled I corresponds to the left hand side. The term labelled II vanishes according to
  \begin{equation}
    \begin{aligned}
      0 ={}& 
      M\left(
      \begin{smallmatrix}
        b & 2 & \dots & k \\
        c_1 & c_2 & \dots & c_{k} \\
      \end{smallmatrix}
      \right)
      =
      \sum_{a=1}^{k}
      (-1)^{a+1}
      M\left(
      \begin{smallmatrix}
        b \\
        c_a \\
      \end{smallmatrix}
      \right)
      M\left(
      \begin{smallmatrix}
        2   & \dots & a       & a+1     & \dots & k \\
        c_1 & \dots & c_{a-1} & c_{a+1} & \dots & c_k \\
      \end{smallmatrix}
      \right)
    \end{aligned}
  \end{equation}
  for $b \in \{2, \dots, k\}$, due to antisymmetry of minors.
\end{proof}
\begin{remark}
  The lemma holds true as stated but the assertion is trivial unless $k \leq \min\{m,n\}$ and all $r$'s, all $c$'s as well as all $d$'s are different.
\end{remark}
\begin{remark}
  A special case of the lemma is the identity
  \begin{equation}
    \begin{aligned}
      M\left(
      \begin{smallmatrix}
        r_{1} & \dots & r_{k} & r_{k+1} \\
        c_{1} & \dots & c_{k} & r_{k+1} \\
      \end{smallmatrix}
      \right)
      M\left(
      \begin{smallmatrix}
        r_{2} & \dots & r_{k+1} \\
        r_{2} & \dots & r_{k+1} \\
      \end{smallmatrix}
      \right)
      = \sum_{a=1}^k (-1)^{a+1}
      M\left(
      \begin{smallmatrix}
        r_{1} & r_{2} & \dots & r_{k+1} \\
        c_{a} & r_{2} & \dots & r_{k+1} \\
      \end{smallmatrix}
      \right)
      M\left(
      \begin{smallmatrix}
        r_{2} & \dots & r_{a} & r_{a+1} & \dots & r_{k} &  r_{k+1} \\
        c_{1} & \dots & c_{a-1} & c_{a+1} & \dots & c_{k} & r_{k+1} \\
      \end{smallmatrix}
      \right)
    \end{aligned}
    \label{eqn:speciallemma1}
  \end{equation}
  which will be used in the proof of lemma \ref{lemma:2}.
\end{remark}

\begin{lemma}
  \label{lemma:2}
  An $m\times n$ matrix M obeys the following identity involving the determinant of a $k\times k$-matrix of minors of $M$
  \begin{equation}
    \begin{aligned}
      &
      \left|
      \begin{array}{cccc}
        M\left(
        \begin{smallmatrix}
          r_1 & r_2 & \dots & r_{k+1} \\
          c_1 & r_2 & \dots & r_{k+1} \\
        \end{smallmatrix}
        \right)
        &
        M\left(
        \begin{smallmatrix}
          r_1 & r_2 & \dots & r_{k+1} \\
          c_2 & r_2 & \dots & r_{k+1} \\
        \end{smallmatrix}
        \right)
        &
        \dots
        &
        M\left(
        \begin{smallmatrix}
          r_1 & r_2 & \dots & r_{k+1} \\
          c_k & r_2 & \dots & r_{k+1} \\
        \end{smallmatrix}
        \right)
        \\
        M\left(
        \begin{smallmatrix}
          r_2 & r_3 & \dots & r_{k+1} \\
          c_1 & r_3 & \dots & r_{k+1} \\
        \end{smallmatrix}
        \right)
        &
        M\left(
        \begin{smallmatrix}
          r_2 & r_3 & \dots & r_{k+1} \\
          c_2 & r_3 & \dots & r_{k+1} \\
        \end{smallmatrix}
        \right)
        &
        \dots
        &
        M\left(
        \begin{smallmatrix}
          r_2 & r_3 & \dots & r_{k+1} \\
          c_k & r_3 & \dots & r_{k+1} \\
        \end{smallmatrix}
        \right)
        \\
        \vdots & \vdots & \ddots & \vdots
        \\
        M\left(
        \begin{smallmatrix}
          r_k & r_{k+1} \\
          c_1 & r_{k+1} \\
        \end{smallmatrix}
        \right)
        &
        M\left(
        \begin{smallmatrix}
          r_k & r_{k+1} \\
          c_2 & r_{k+1} \\
        \end{smallmatrix}
        \right)
        &
        \dots
        &
        M\left(
        \begin{smallmatrix}
          r_k & r_{k+1} \\
          c_k & r_{k+1} \\
        \end{smallmatrix}
        \right)
      \end{array}
      \right|
      = \\
      ={}&
      M\left(
      \begin{smallmatrix}
        r_1 & r_{2} & \dots & r_k & r_{k+1} \\
        c_1 & c_{2} & \dots & c_k & r_{k+1} \\
      \end{smallmatrix}
      \right)
      M\left(
      \begin{smallmatrix}
        r_2 & \dots & r_{k+1} \\
        r_2 & \dots & r_{k+1} \\
      \end{smallmatrix}
      \right)
      \dots
      M\left(
      \begin{smallmatrix}
        r_{k} & r_{k+1} \\
        r_{k} & r_{k+1} \\
      \end{smallmatrix}
      \right)
    \end{aligned}
  \end{equation}
  where $r_i \in \{1, \dots, m\}$ and $c_i \in \{1, \dots, n\}$ and $k \in \nats$.
\end{lemma}
\begin{proof}
  The proof works by induction. The base case $k=1$ trivial. Assume that the formula holds for $k \leq q-1$ for some $q-1\in\nats$. Expanding the determinant for $k = q$ along the first row and using the induction hypothesis on the remaining $(q-1)\times(q-1)$-determinants gives
  \begin{equation}
    \begin{aligned}
      \mathrm{LHS} ={}&
      \sum_{a=1}^q
      (-1)^{a+1}
      M\left(
      \begin{smallmatrix}
        r_1 & r_{2} & \dots & r_{q+1} \\
        c_a & r_{2} & \dots & r_{q+1} \\
      \end{smallmatrix}
      \right)
      M\left(
      \begin{smallmatrix}
        r_2 & \dots & r_a & r_{a+1} & \dots & r_q & r_{q+1} \\
        c_1 & \dots & c_{a-1} & c_{a+1} & \dots & c_q & r_{q+1} \\
      \end{smallmatrix}
      \right)
      M\left(
      \begin{smallmatrix}
        r_3 & \dots & r_{q+1} \\
        r_3 & \dots & r_{q+1} \\
      \end{smallmatrix}
      \right)
      \dots
      M\left(
      \begin{smallmatrix}
        r_{q} & r_{q+1} \\
        r_{q} & r_{q+1} \\
      \end{smallmatrix}
      \right)
      = \\
      ={}& 
      M\left(
      \begin{smallmatrix}
        r_1 & \dots & r_{q} & r_{q+1} \\
        c_1 & \dots & c_{q} & r_{q+1} \\
      \end{smallmatrix}
      \right)
      M\left(
      \begin{smallmatrix}
        r_2 & \dots & r_{q+1} \\
        r_2 & \dots & r_{q+1} \\
      \end{smallmatrix}
      \right)
      M\left(
      \begin{smallmatrix}
        r_3 & \dots & r_{q+1} \\
        r_3 & \dots & r_{q+1} \\
      \end{smallmatrix}
      \right)
      \dots
      M\left(
      \begin{smallmatrix}
        r_{q} & r_{q+1} \\
        r_{q} & r_{q+1} \\
      \end{smallmatrix}
      \right)
    \end{aligned}
  \end{equation}
  where we have used \eqref{eqn:speciallemma1}. Peano's axiom of induction now establishes the lemma.
\end{proof}

We now have what we need to prove a formula for the norms of the dilatons.
\begin{thm}[Norms of the dilatons of an $\SL( n, \rats_p )$-matrix]
  \label{thm:piwasawa}
  The norms of the dilatons in the Iwasawa-decomposition of a matrix $M \in \SL( n, \rats_p )$ are given by
  \begin{equation}
    \left| y_{n-k} \right|_p =
    \left(
    \max_{\sigma \in \Theta_{k}^{n}}
    \left\{ \left|
      M\left(
      \begin{smallmatrix}
        n-k+1 & \dots & n\\
        \sigma(1) & \dots & \sigma(k) \\
      \end{smallmatrix}
      \right)
    \right|_p \right\}
    \right)^{-1}
    \label{eqn:min}
    \qtextq{where}
    k \in \{1, \dots, n-1\}
  \end{equation}
  and $\Theta_k^{n}$ detones the set of all ordered subsets of $\{1, \dots, n\}$ of order $k$.\footnote{In words: The norm of the dilaton $y_{n-k}$ is the inverse of the norm of the largest anti-leading principal minor of order $k$. The formula produces the desired result $y_0 = y_n = 1$ for $k=n$ and $k=0$ respectively.}
  
  Alternatively in terms of the generalized Pl\"ucker coordinates $\rho_k$,
  \begin{equation}
    \left| y_{n-k} \right|_p =
    \left(
    \max_{x \in \rho_k}
    \left\{ \left|
      x
    \right|_p \right\}
    \right)^{-1}
    \qtextq{where}
    k \in \{1, \dots, n-1\}
    .
  \end{equation}
\end{thm}
\begin{proof}
  The proof works by induction.
  %
  Suppose that the formula holds up to and including $\SL( n-1, \rats_p )$ and consider $\SL( n, \rats_p )$.

  Restrict to the case $\left| M_{na} \right|_p \geq \left| M_{ni} \right|_p$ for $i\in \{1, \dots, n\}$, i.e.\ the element with the largest $p$-adic norm sits in column $a$. If there is no unique such element, any one of the largest elements on the bottom row of $M$ may play the role of $M_{na}$. Note that $M_{na} \neq 0$ since otherwise the bottom row would be a zero-row, rendering $M$ singular. Performing a strong UL-decomposition on $M$ where we move column $a$ to the rightmost place gives 
  \begin{equation}
    M = V \Delta \Lambda \Pi_a^{-1} = V \Delta
    \left(
    \begin{array}{ccccc}
      1 \\
      \tilde{M}\left(
      \begin{smallmatrix}
        2 & 3 & \dots & n \\
        1 & 3 & \dots & n \\
      \end{smallmatrix}
      \right)
      \eta_{1}
      & 1
      \\
      \vdots & \vdots & \ddots \\
      \tilde{M}\left(
      \begin{smallmatrix}
        n-1 & n \\
        1 & n\\
      \end{smallmatrix}
      \right)
      \eta_{n-2}
      &
      \tilde{M}\left(
      \begin{smallmatrix}
        n-1 & n \\
        2 & n \\
      \end{smallmatrix}
      \right)
      \eta_{n-2}
      & \dots &
      1
      \\
      \tilde{M}\left(
      \begin{smallmatrix}
        n \\
        1 \\
      \end{smallmatrix}
      \right)
      \eta_{n-1}
      &
      \tilde{M}\left(
      \begin{smallmatrix}
        n \\
        2 \\
      \end{smallmatrix}
      \right)
      \eta_{n-1}
      & \dots &
      \tilde{M}\left(
      \begin{smallmatrix}
        n \\
        n-1 \\
      \end{smallmatrix}
      \right)
      \eta_{n-1}
      &
      1
    \end{array}
    \right)
    \Pi_a^{-1}
  \end{equation}
  where we have denoted $\tilde{M} = M \Pi_a$. Note that $\Pi_a^{-1} \in \SL( n, \ints_p )$. Note furthermore that the bottom row in the matrix $\Lambda$ is simply a permutation of the bottom row of $M$ divided by $
  \tilde{\eta}_{n-1}^{-1}
  =
  \tilde{M}\left(
  \begin{smallmatrix}
    n \\
    n \\
  \end{smallmatrix}
  \right)
  =
  M_{na}
  $. Because $M_{na}$ is assumed to have the largest $p$-adic norm, we get that every element in the bottom row of $\Lambda$ is a $p$-adic integer. Thanks to the unit lower triangular form, the bottom row easily factorizes out on the right and we are left with
  \begin{equation}
    \label{eqn:extractedrow}
    M = V \Delta
    \left(
    \begin{array}{ccccc}
      1 \\
      \tilde{M}\left(
      \begin{smallmatrix}
        2 & 3 & \dots & n \\
        1 & 3 & \dots & n \\
      \end{smallmatrix}
      \right)
      \eta_{1}
      & 1
      \\
      \vdots & \vdots & \ddots \\
      \tilde{M}\left(
      \begin{smallmatrix}
        n-1 & n \\
        1 & n\\
      \end{smallmatrix}
      \right)
      \eta_{n-2}
      &
      \tilde{M}\left(
      \begin{smallmatrix}
        n-1 & n \\
        2 & n \\
      \end{smallmatrix}
      \right)
      \eta_{n-2}
      & \dots &
      1
      \\
      0
      &
      0
      & \dots &
      0
      &
      1
    \end{array}
    \right)
    R \Pi_a^{-1}
  \end{equation}
  where $R \in \SL( n, \ints_p )$ contains the bottom row of $\Lambda$. The block diagonal form implies that there will be no further contributions to the dilaton $y_{n-1}$ and its norm is now fixed at
  \begin{equation}
    \left| y_{n-1} \right|_p = \left| \eta_{n-1} \right|_p =
    \left|
    \frac{1}{
      \tilde{M}\left(
      \begin{smallmatrix}
        n \\
        n \\
      \end{smallmatrix}
      \right)
    }
    \right|_p
    =
    \left|
    M\left(
    \begin{smallmatrix}
      n \\
      a \\
    \end{smallmatrix}
    \right)
    \right|_p^{-1}
    =
    \left(
    \max_{\sigma \in \Theta_{1}^{n}}
    \left\{ \left|
      M\left(
      \begin{smallmatrix}
        n\\
        \sigma(1) \\
      \end{smallmatrix}
      \right)
    \right|_p \right\}
    \right)^{-1}
    ,
  \end{equation}
  again using the fact that the element $M_{na}$ has the largest $p$-adic norm. This expression is of the form \eqref{eqn:min}. Putting $n = 2$ here proves the base case $\SL(2, \rats_p)$. Next we treat the dilatons $y_{1}, \dots, y_{n-2}$.

  Note that we can write
  \begin{equation}
    \begin{aligned}
      &
      \left(
      \begin{array}{cccc}
        1 \\
        \tilde{M}\left(
        \begin{smallmatrix}
          2 & 3 & \dots & n \\
          1 & 3 & \dots & n \\
        \end{smallmatrix}
        \right)
        \eta_{1}
        & 1
        \\
        \vdots & \vdots & \ddots \\
        \tilde{M}\left(
        \begin{smallmatrix}
          n-1 & n \\
          1 & n\\
        \end{smallmatrix}
        \right)
        \eta_{n-2}
        &
        \tilde{M}\left(
        \begin{smallmatrix}
          n-1 & n \\
          2 & n \\
        \end{smallmatrix}
        \right)
        \eta_{n-2}
        & \dots &
        1
      \end{array}
      \right)
      =
      \\
      ={}&
      \left(
      \begin{array}{cccc}
        \tilde{M}\left(
        \begin{smallmatrix}
          1 & 2 & \dots & n \\
          1 & 2 & \dots & n \\
        \end{smallmatrix}
        \right)
        \eta_{0}
        &
        \tilde{M}\left(
        \begin{smallmatrix}
          1 & 2 & \dots & n \\
          2 & 2 & \dots & n \\
        \end{smallmatrix}
        \right)
        \eta_{0}
        &
        \dots
        &
        \tilde{M}\left(
        \begin{smallmatrix}
          1 & 2 & \dots & n \\
          n-1 & 2 & \dots & n \\
        \end{smallmatrix}
        \right)
        \eta_{0}
        \\
        \tilde{M}\left(
        \begin{smallmatrix}
          2 & 3 & \dots & n \\
          1 & 3 & \dots & n \\
        \end{smallmatrix}
        \right)
        \eta_{1}
        &
        \tilde{M}\left(
        \begin{smallmatrix}
          2 & 3 & \dots & n \\
          2 & 3 & \dots & n \\
        \end{smallmatrix}
        \right)
        \eta_{1}
        &
        \dots
        &
        \tilde{M}\left(
        \begin{smallmatrix}
          2 & 3 & \dots & n \\
          n-1 & 3 & \dots & n \\
        \end{smallmatrix}
        \right)
        \eta_{1}
        \\
        \vdots & \vdots & \ddots & \vdots \\
        \tilde{M}\left(
        \begin{smallmatrix}
          n-1 & n \\
          1 & n\\
        \end{smallmatrix}
        \right)
        \eta_{n-2}
        &
        \tilde{M}\left(
        \begin{smallmatrix}
          n-1 & n \\
          2 & n \\
        \end{smallmatrix}
        \right)
        \eta_{n-2}
        & \dots &
        \tilde{M}\left(
        \begin{smallmatrix}
          n-1 & n \\
          n-1 & n \\
        \end{smallmatrix}
        \right)
        \eta_{n-2}
      \end{array}
      \right)
      .
    \end{aligned}
  \end{equation}
  We apply the induction hypothesis to this $(n-1) \times (n-1)$-diagonal block of \eqref{eqn:extractedrow} and compute its contribution to the norm of the dilaton $y_{n-(k+1)}$\footnote{We choose to consider $y_{n-(k+1)} = y_{n-1-k}$ so that we get nice expressions in what follows. Later we send $k$ to $k-1$ to compare with \eqref{eqn:min}.} for $k\in\{1, \dots n-1\}$. The formula \eqref{eqn:min} implies that we need to compute minors like
  \begin{equation}
    \left|
    \begin{smallmatrix}
      \tilde{M}\left(
      \begin{smallmatrix}
        n-k & n-k+1 & \dots & n \\
        \tilde{\sigma}(1) & n-k+1 & \dots & n \\
      \end{smallmatrix}
      \right)
      \eta_{n-k-1}
      &
      \tilde{M}\left(
      \begin{smallmatrix}
        n-k & n-k+1 & \dots & n \\
        \tilde{\sigma}(2) & n-k+1 & \dots & n \\
      \end{smallmatrix}
      \right)
      \eta_{n-k-1}
      &
      \dots
      &
      \tilde{M}\left(
      \begin{smallmatrix}
        n-k & n-k+1 & \dots & n \\
        \tilde{\sigma}(k) & n-k+1 & \dots & n \\
      \end{smallmatrix}
      \right)
      \eta_{n-k-1}
      \\
      \tilde{M}\left(
      \begin{smallmatrix}
        n-k-1 & n-k-2 & \dots & n \\
        \tilde{\sigma}(1) & n-k-2 & \dots & n \\
      \end{smallmatrix}
      \right)
      \eta_{n-k-2}
      &
      \tilde{M}\left(
      \begin{smallmatrix}
        n-k-1 & n-k-2 & \dots & n \\
        \tilde{\sigma}(2) & n-k-2 & \dots & n \\
      \end{smallmatrix}
      \right)
      \eta_{n-k-2}
      &
      \dots
      &
      \tilde{M}\left(
      \begin{smallmatrix}
        n-k-1 & n-k-2 & \dots & n \\
        \tilde{\sigma}(k) & n-k-2 & \dots & n \\
      \end{smallmatrix}
      \right)
      \eta_{n-k-2}
      \\
      \vdots & \vdots & \ddots & \vdots \\
      \tilde{M}\left(
      \begin{smallmatrix}
        n-1 & n \\
        \tilde{\sigma}(1) & n\\
      \end{smallmatrix}
      \right)
      \eta_{n-2}
      &
      \tilde{M}\left(
      \begin{smallmatrix}
        n-1 & n \\
        \tilde{\sigma}(2) & n \\
      \end{smallmatrix}
      \right)
      \eta_{n-2}
      & \dots &
      \tilde{M}\left(
      \begin{smallmatrix}
        n-1 & n \\
        \tilde{\sigma}(k) & n \\
      \end{smallmatrix}
      \right)
      \eta_{n-2}
    \end{smallmatrix}
    \right|
  \end{equation}
  where $\tilde{\sigma} \in \Theta_{k}^{n-1}$. Using lemma \ref{lemma:2}, this evaluates to
  \begin{equation}
    \begin{aligned}
      &
      \tilde{M}\left(
      \begin{smallmatrix}
        n-k & \dots & n-1 & n \\
        \tilde{\sigma}(1) & \dots & \tilde{\sigma}(k) & n \\
      \end{smallmatrix}
      \right)
      \tilde{M}\left(
      \begin{smallmatrix}
        n-k+1 & \dots & n \\
        n-k+1 & \dots & n \\
      \end{smallmatrix}
      \right)
      \dots
      \tilde{M}\left(
      \begin{smallmatrix}
        n-1 & n \\
        n-1 & n \\
      \end{smallmatrix}
      \right)
      \eta_{n-k-1}
      \dots
      \eta_{n-2}
      = \\
      ={} &
      \tilde{M}\left(
      \begin{smallmatrix}
        n-k & \dots & n-1 & n \\
        \tilde{\sigma}(1) & \dots & \tilde{\sigma}(k) & n \\
      \end{smallmatrix}
      \right)
      \tilde{M}\left(
      \begin{smallmatrix}
        n-k & \dots & n \\
        n-k & \dots & n \\
      \end{smallmatrix}
      \right)
    \end{aligned}
    .
  \end{equation}
  The dilaton $y_{n-(k+1)}$ already has a contribution $
  \eta_{n-(k+1)} =
  \left(
  \tilde{M}\left(
  \begin{smallmatrix}
    n-k & \dots & n \\
    n-k & \dots & n \\
  \end{smallmatrix}
  \right)
  \right)^{-1}
  $ from the matrix $\Delta$. Multiplying these contributions together and considering all subsets $\tilde{\sigma} \in \Theta_{k}^{n-1}$ gives the final answer for the norm of the dilaton $y_{n-(k+1)}$ in the case $\left| M_{na}\right|_p \geq \left| M_{ni} \right|_p$ as
  \begin{equation}
    \left| y_{n-(k+1)} \right|_p =
    \left(
    \max_{\tilde{\sigma} \in \Theta_{k}^{n-1}}
    \left\{ \left|
      \tilde{M}\left(
      \begin{smallmatrix}
        n-k & \dots & n-1 & n\\
        \tilde{\sigma}(1) & \dots & \tilde{\sigma}(k) & n \\
      \end{smallmatrix}
      \right)
    \right|_p \right\}
    \right)^{-1}
    .
  \end{equation}
  Note that the exact form of the permutation matrix $\Pi_a$ becomes irrelevant, since the expression above anyway considers all order $k$ subsets of the columns $1$ through $n-1$ in $\tilde{M}$. Note furthermore that any minus signs present in $\Pi_a$ as discussed in footnote \ref{footnote:minussigns} make no difference since the minors sit inside a norm.

  We can think of the order $k$-subsets $\tilde{\sigma}\in \Theta_{k}^{n-1}$ as order $k+1$ subsets $\sigma\in \Theta{}_{k+1}^{n}$ restricted such that $\sigma(k+1) = n$. Writing the equation above in terms of $M$ and replacing $k$ by $k-1$ gives
  \begin{equation}
    \left| y_{n-k} \right|_p =
    \left(
    \max_{\sigma \in \Theta'{}_{k}^{n}}
    \left\{ \left|
      M\left(
      \begin{smallmatrix}
        n-k+1 & \dots & n\\
        \sigma(1) & \dots & \sigma(k) \\
      \end{smallmatrix}
      \right)
    \right|_p \right\}
    \right)^{-1}
    \label{eqn:specialmin}
  \end{equation}
  where the prime on $\Theta'{}_k^n$ indicates that all subsets must contain $a$. To finalize the proof, we must show that \eqref{eqn:min} reduces to \eqref{eqn:specialmin} in the case $\left| M_{na}\right|_p \geq \left| M_{ni} \right|_p$. This is equivalent to showing that for every anti-leading minor which doesn't include column $a$, one can find an anti-leading minor which does include column $a$ and whose $p$-adic norm is at least as large as the original minor. To see that this is true, pick an anti-leading minor of order $k$ which doesn't contain column $a$ and consider its norm\footnote{By operating inside of a norm, we don't need to keep track of the overall sign.}
  \begin{align}
    \label{eqn:candidateminor}
    \left|
    M\left(
    \begin{smallmatrix}
      n-k+1 & \dots & n\\
      \sigma_a(1) & \dots & \sigma_a(k) \\
    \end{smallmatrix}
    \right)
    \right|_p
    =
    \left|
    \sum_{i=1}^{k}
    (-1)^{i}
    M\left(
    \begin{smallmatrix}
      n \\
      \sigma_a(i) \\
    \end{smallmatrix}
    \right)
    M\left(
    \begin{smallmatrix}
      n-k+1 & \dots & n-k+i-1 & n-k+i & \dots & n-1 \\
      \sigma_a(1) & \dots & \sigma_a(i-1) & \sigma_a(i+1) & \dots & \sigma_a(k) \\
    \end{smallmatrix}
    \right)
    \right|_p
  \end{align}
  where $\sigma_a \in \Theta_k^n$ and $a \notin \sigma_a$. Define $(-1)^k u_i$ as the minor under consideration but with $\sigma_a(i)$ replaced by $a$. This minor is then already present in \eqref{eqn:specialmin}. Laplace expansion gives
  \begin{equation}
    \begin{aligned}
      {}& (-1)^{k} u_i \equiv 
      M\left(
      \begin{smallmatrix}
        n-k+1 & \dots & n-k+i-1 & n-k+i & n-k+i+1 & \dots & n \\
        \sigma_a(1) & \dots & \sigma_a(i-1) & a & \sigma_a(i+1) & \dots & \sigma_a(k) \\
      \end{smallmatrix}
      \right) = \\
      ={}&
      \sum_{j=1}^{i-1}
      (-1)^{k+j}
      M\left(
      \begin{smallmatrix}
        n \\
        \sigma_a(j) \\
      \end{smallmatrix}
      \right)
      M\left(
      \begin{smallmatrix}
        n-k+1 & \dots & n-k+j-1 & n-k+j & \dots & n-k+i-2 & n-k+i-1 & n-k+i & \dots & n-1 \\
        \sigma_a(1) & \dots & \sigma_a(j-1) & \sigma_a(j+1) & \dots & \sigma_a(i-1) & a & \sigma_a(i+1) & \dots & \sigma_a(k) \\
      \end{smallmatrix}
      \right) \\
      +{}&(-1)^{k+i}
      M\left(
      \begin{smallmatrix}
        n \\
        a \\
      \end{smallmatrix}
      \right)
      M\left(
      \begin{smallmatrix}
        n-k+1 & \dots & n-k+i-1 & n-k+i & \dots & n-1 \\
        \sigma_a(1) & \dots & \sigma_a(i-1) & \sigma_a(i+1) & \dots & \sigma_a(k) \\
      \end{smallmatrix}
      \right) \\
      +{}&
      \sum_{j=i+1}^{k}
      (-1)^{k+j}
      M\left(
      \begin{smallmatrix}
        n \\
        \sigma_a(j) \\
      \end{smallmatrix}
      \right)
      M\left(
      \begin{smallmatrix}
        n-k+1 & \dots & n-k+i-1 & n-k+i & n-k+i+1 & \dots & n-k+j-1 & n-k+j & \dots & n-1 \\
        \sigma_a(1) & \dots & \sigma_a(i-1) & a & \sigma_a(i+1) & \dots & \sigma_a(j-1) & \sigma_a(j+1) & \dots & \sigma_a(k) \\
      \end{smallmatrix}
      \right)
      .
    \end{aligned}
    \label{eqn:ni}
  \end{equation}
  We now show that the norm of the minor under consideration \eqref{eqn:candidateminor} can be expressed in terms of the norm of a sum of $u$'s as
  \begin{equation}
    \left|
    M\left(
    \begin{smallmatrix}
      n-k+1 & \dots & n\\
      \sigma_a(1) & \dots & \sigma_a(k) \\
    \end{smallmatrix}
    \right)
    \right|_p
    =
    \left|
    \frac{1}{
      M\left(
      \begin{smallmatrix}
        n \\
        a \\
      \end{smallmatrix}
    \right)}
    \sum_{i=1}^{k} 
    M\left(
    \begin{smallmatrix}
      n \\
      \sigma_a(i) \\
    \end{smallmatrix}
    \right)
    u_i
    \right|_p
    \label{eqn:cancelation}
    .
  \end{equation}
  By making recourse with \eqref{eqn:candidateminor}, it is easy to see that the terms in the right hand side of \eqref{eqn:candidateminor} come from the second line in the right hand side of \eqref{eqn:ni}. To see that the remaining terms cancel, consider the term inside the sum in the right hand side of \eqref{eqn:cancelation} with $i = I$ and $j = J$ where $J<I$ (coming from the first line of the right hand side of \eqref{eqn:ni})
  \begin{equation}
    M\left(
    \begin{smallmatrix}
      n \\
      \sigma_a(I) \\
    \end{smallmatrix}
    \right)
    (-1)^{J}
    M\left(
    \begin{smallmatrix}
      n \\
      \sigma_a(J) \\
    \end{smallmatrix}
    \right)
    M\left(
    \begin{smallmatrix}
      n-k+1 & \dots & n-k+J-1 & n-k+J & \dots & n-k+I-2 & n-k+I-1 & n-k+I & \dots & n-1 \\
      \sigma_a(1) & \dots & \sigma_a(J-1) & \sigma_a(J+1) & \dots & \sigma_a(I-1) & a & \sigma_a(I+1) & \dots & \sigma_a(k) \\
    \end{smallmatrix}
    \right)
    .
  \end{equation}
  This cancels with the term $i=J$ and $j=I$ (coming from the third line of the right hand side of \eqref{eqn:ni}) as is seen by writing
  \begin{equation}
    \begin{aligned}
      {}&
      M\left(
      \begin{smallmatrix}
        n \\
        \sigma_a(J) \\
      \end{smallmatrix}
      \right)
      (-1)^{I}
      M\left(
      \begin{smallmatrix}
        n \\
        \sigma_a(I) \\
      \end{smallmatrix}
      \right)\\
      {}&
      M\left(
      \begin{smallmatrix}
        n-k+1 & \dots & n-k+J-1 & n-k+J & n-k+J+1 & \dots & n-k+I-1 & n-k+I & \dots & n-1 \\
        \sigma_a(1) & \dots & \sigma_a(J-1) & a & \sigma_a(J+1) & \dots & \sigma_a(I-1) & \sigma_a(I+1) & \dots & \sigma_a(k) \\
      \end{smallmatrix}
      \right) = \\
      ={}&
      M\left(
      \begin{smallmatrix}
        n \\
        \sigma_a(J) \\
      \end{smallmatrix}
      \right)
      (-1)^{-J+1}
      M\left(
      \begin{smallmatrix}
        n \\
        \sigma_a(I) \\
      \end{smallmatrix}
      \right)\\
      {}&
      M\left(
      \begin{smallmatrix}
        n-k+1 & \dots & n-k+J-1 & n-k+J & \dots & n-k+I-2 & n-k+I-1 & n-k+I & \dots & n-1 \\
        \sigma_a(1) & \dots & \sigma_a(J-1) & \sigma_a(J+1) & \dots & \sigma_a(I-1) & a & \sigma_a(I+1) & \dots & \sigma_a(k) \\
      \end{smallmatrix}
      \right)
      .
    \end{aligned}
  \end{equation}
  This establishes \eqref{eqn:cancelation}. The ultrametric property of the $p$-adic norm now gives that the norm of the minor under consideration \eqref{eqn:cancelation} is dominated by the norm of the largest $u$
  \begin{equation}
    \left|
    \sum_{i=1}^{k} 
    \frac{
      M\left(
      \begin{smallmatrix}
        n \\
        \sigma_a(i) \\
      \end{smallmatrix}
      \right)
    }{
      M\left(
      \begin{smallmatrix}
        n \\
        a \\
      \end{smallmatrix}
    \right)}
    u_i
    \right|_p
    \leq
    \max \left\{
      \left|
      \frac{M_{n, \sigma_a(1)}}{M_{na}}
      u_1
      \right|_p 
      , \dots, 
      \left|
      \frac{M_{n, \sigma_a(k)}}{M_{na}}
      u_k
      \right|_p 
    \right\}
    \leq
    \max \left\{
      \left|
      u_1
      \right|_p 
      , \dots, 
      \left|
      u_k
      \right|_p 
    \right\}
  \end{equation}
  where we have used $\left| M_{na}\right|_p \geq \left| M_{ni} \right|_p$. Since the $u$'s are all present in \eqref{eqn:specialmin}, we can safely include the minor under consideration as its would never be picked over the $u$'s. This argument holds for all anti-leading minors of order $k$ and we may threfore include them all in the right hand side of \eqref{eqn:specialmin}, successfully reproducing \eqref{eqn:min}. Peano's axiom of induction now establishes the theorem.
\end{proof}
\begin{remark}
  The proof shows that in the case $\left| M_{na}\right|_p \geq \left| M_{ni} \right|_p$, the expression \eqref{eqn:min} reduces to \eqref{eqn:specialmin}.
\end{remark}
\begin{remark}
  The method of using strong UL-decomposition to reduce the problem of Iwasawa-decomposition of an $n\times n$-matrix to that of an $(n-1)\times (n-1)$-matrix can be iterated and defines an algorithm for a complete Iwasawa-decomposition $M = NAK$ of any matrix $M \in \SL( n, \rats_p )$. All other Iwasawa-decompositions can then be found by using theorem \ref{thm:nonuniqueness}. Using this method to derive a general formula for the matrix elements of $N$ and $A$ (and hence $K$) given $M$ seems feasible. Such a formula is expected to express also the axions in terms of a conditional clause like $\max$. It this endeavour, it would probably be convenient to use theorem \ref{thm:nonuniqueness} to impose some standard normalization on the axions and dilatons, for instance one may normalize the dilatons to be just pure powers of $p$.
\end{remark}

\section{Decomposing $\SL( n, \reals )$}
Throughout this section we take the liberty of using a wide array of indices, all of which range from $1$ to $n$. We employ the summation convention for repeated indices provided that one stands upstairs and the other downstairs.

We begin with a more precise statement of Iwasawa-decomposition for matrices in $\SL( n, \reals )$.
\begin{thm}[Iwasawa-decomposition for $\SL( n, \reals )$]
  A matrix $M \in \SL( n, \reals )$ may be written uniquely as
  \begin{equation}
    M = NAK
    \label{eqn:realiwasawa}
  \end{equation}
  where $N \in \SL( n, \reals )$ is unit upper triangular, $A \in \SL( n, \reals )$ is diagonal with positive entries and $K \in \SO( n )$. Furthermore, denoting the row-vectors in $M$ by $V_i$, $i \in \{1, \dots, n \}$, and parameterizing $N$ and $A$ as
  \begin{equation}
    N_{ij} =
    \begin{cases}
      1, & i = j \\
      x_{ij}, & i < j \\
      0, & i > j
    \end{cases}
    \qtextq{and}
    A_{ij} = \frac{y_i}{y_{i-1}} \delta_{ij}
    \qtextq{with}
    y_0 \equiv y_n \equiv 1
    ,
  \end{equation}
  we have that the axions and dilatons are given by
  \begin{align}
    x_{\mu \nu} ={} & y_{\nu-1}^2 \epsilon\left( V_\mu, V_{\nu+1}, \dots, V_n; V_{\nu}, V_{\nu+1}, \dots, V_n \right), \quad \mu < \nu, \label{eqn:axions} \qtext{and} \\
    y_{\mu}^{-2} ={} & \epsilon\left( V_{\mu+1}, \dots, V_n; V_{\mu+1}, \dots, V_n \right) \label{eqn:realdilatons}
  \end{align}
  where $\epsilon$ denotes the totally antisymmetric product
  \begin{equation}
	\begin{aligned}
      \epsilon\left( A_1, \dots, A_m; B_1, \dots, B_m \right) = \delta_{a_{1} \dash a_m}^{i_{1} \dash i_m} \left( A_{1} \right)^{a_{1}} \dots{} \left( V_m \right)^{a_m} \left( B_{1} \right)_{i_{1}} \dots{} \left( B_m \right)_{i_m}
	\end{aligned}
  \end{equation}
  where the $A$'s and $B$'s are $n$-vectors and 
  \begin{align}
    \delta_{a_{1} \dash a_m}^{i_{1} \dash i_m} = m! \delta_{\left[ a_{1} \right.}^{i_{1}} \dots \delta_{\left. a_{m} \right]}^{i_{m}} = \frac{1}{(n-m)!} \epsilon_{a_1 \dash a_m \alpha_{m+1} \dash \alpha_n} \epsilon^{i_1 \dash i_m \alpha_{m+1} \dash \alpha_n}
  \end{align}
  denotes the generalized Kronecker delta.
\end{thm}
The procedure of writing a real matrix $M$ in Iwasawa-form is tantamount to Gram-Schmidt orthogonalization of the $n$ row-vectors in $M$ for which there are recursive formulae. The orthogonal matrix $K$ consists of $n$ orthonormal row-vectors and the unit upper triangular matrix $N$ together with the normalization in $A$ then specifies the appropriate linear combinations of these row-vectors to build the row-vectors in $M$. Oftentimes in the literature, people denote the product of $A$ and $N$ as $R$ and speak about the QR-decomposition\footnote{This is in the case $M = KAN$, where $K=Q$ and $AN=R$.}.

A very quick way to arrive at the non-recursive formulae \eqref{eqn:axions} and \eqref{eqn:realdilatons} given above is by means of the UL-decomposition as done in \cite{householder2013theory}. The argument goes like this: Write $MM^{\T} = N A^2 N^{T}$. The right hand side is then a UL-decomposition of $M M^{\T}$ and the matrix elements of $A^2$ and $N$ must then be given by \eqref{eqn:etas} and \eqref{eqn:vs} respectively. As a complement to this method, we give a proof for the formulae above which doesn't rely on the UL-decomposition.

\begin{proof}
  We have the equality
  \begin{align}
    M M^{\T} = N A^2 N^{\T}
    \label{eqn:mm}
    .
  \end{align}
  To illustrate the idea behind the proof, we write out the right hand side explicitly for the case $n=4$
  \begin{align}
    \left(
    \begin{array}{cccc}
     y_{1}^2+\frac{x_{12}^2 y_{2}^2}{y_{1}^2}+\frac{x_{14}^2}{y_{3}^2}+\frac{x_{13}^2 y_{3}^2}{y_{2}^2} & \frac{x_{12} y_{2}^2}{y_{1}^2}+\frac{x_{14} x_{24}}{y_{3}^2}+\frac{x_{13} x_{23} y_{3}^2}{y_{2}^2} & \frac{x_{14} x_{34}}{y_{3}^2}+\frac{x_{13} y_{3}^2}{y_{2}^2} & \frac{x_{14}}{y_{3}^2} \\
     \frac{x_{12} y_{2}^2}{y_{1}^2}+\frac{x_{14} x_{24}}{y_{3}^2}+\frac{x_{13} x_{23} y_{3}^2}{y_{2}^2} & \frac{y_{2}^2}{y_{1}^2}+\frac{x_{24}^2}{y_{3}^2}+\frac{x_{23}^2 y_{3}^2}{y_{2}^2} & \frac{x_{24} x_{34}}{y_{3}^2}+\frac{x_{23} y_{3}^2}{y_{2}^2} & \frac{x_{24}}{y_{3}^2} \\
     \frac{x_{14} x_{34}}{y_{3}^2}+\frac{x_{13} y_{3}^2}{y_{2}^2} & \frac{x_{24} x_{34}}{y_{3}^2}+\frac{x_{23} y_{3}^2}{y_{2}^2} & \frac{x_{34}^2}{y_{3}^2}+\frac{y_{3}^2}{y_{2}^2} & \frac{x_{34}}{y_{3}^2} \\
     \frac{x_{14}}{y_{3}^2} & \frac{x_{24}}{y_{3}^2} & \frac{x_{34}}{y_{3}^2} & \frac{1}{y_{3}^2} \\
    \end{array}
    \right)
  \end{align}
  Starting from the $(4, 4)$ entry and working ``backwards'', i.e.\ proceeding as $(4, 4) \rightarrow (3, 4) \rightarrow (3, 3) \rightarrow (2, 4) \rightarrow \dots$, we notice that each equation is solvable in terms of variables that have previously been determined. The $(\mu, \mu)$-equation allows for determination of $y_{\mu-1}$ and the $(\mu, \nu)$ ($\mu < \nu$) allows for determination of $x_{\mu \nu}$ all in terms of known variables. We now carry this out for the general case.

  Matrix elements of the left- and right hand sides of \cref{eqn:mm} evaluate to
  \begin{align}
    \left( M M^{\T} \right)_{\mu \nu} = V_{\mu} \cdot V_{\nu} = \left( V_\mu \right)_A \left( V_\nu \right)^I \delta_I^A = \epsilon\left( V_{\mu}; V_{\nu} \right)
  \end{align}
  and (assuming $\mu < \nu$)
  \begin{align}
    \begin{aligned}
      \left( NA^2 N^{\T} \right)_{\mu \nu} ={}& \sum_{r=1}^n \sum_{s=1}^n N_{\mu r} \left( A^2 \right)_{rs} \left( N^{\T} \right)_{s \nu} = \sum_{r=1}^n \sum_{s=1}^n N_{\mu r} \frac{y_r^2}{y_{r-1}^2} \delta_{rs} N_{\nu s} = \\
      ={} & \sum_{r=1}^n N_{\mu r} \frac{y_r^2}{y_{r-1}^2} N_{\nu r} = \sum_{r=\nu}^n N_{\mu r} \frac{y_r^2}{y_{r-1}^2} N_{\nu r} \\
      ={} & x_{\mu \nu} \frac{y_\nu^2}{y_{\nu-1}^2} + \sum_{r=\nu+1}^n x_{\mu r} \frac{y_r^2}{y_{r-1}^2} x_{\nu r}
    \end{aligned}
  \end{align}
  respectively. Solving for $x_{\mu \nu}$ gives
  \begin{align}
    x_{\mu \nu} = \frac{y_{\nu-1}^2}{y_{\nu}^2} \left( \epsilon\left( V_{\mu}; V_{\nu}  \right) - \sum_{r=\nu+1}^n \frac{y_r^2}{y_{r-1}^2} x_{\mu r} x_{\nu r} \right)
    \label{eqn:collapsingsum}
    .
  \end{align}
  We assume that all $y_{\rho}$ for $\rho \geq \mu$, and $x_{\rho \sigma}$ for $\rho > \mu$, and $x_{\mu \sigma}$ for $\sigma > \nu$ have been found, and are of the form in \cref{eqn:realdilatons,eqn:axions}. The sum telescopes through the identity
  \begin{align}
    \frac{\epsilon\left( V_{\mu}, V_{r+1}, \dots, V_n; V_{\nu}, V_{r+1}, \dots, V_n \right)}{\epsilon\left( V_{r+1}, \dots, V_n; V_{r+1}, \dots, V_n \right)} - \frac{y_r^2}{y_{r-1}^2} x_{\mu r} x_{\nu r} = \frac{\epsilon\left( V_{\mu}, V_{r}, \dots, V_n; V_{\nu}, V_{r}, \dots, V_n \right)}{\epsilon\left( V_{r}, \dots, V_n; V_{r}, \dots, V_n \right)}
    \label{eqn:telescope}
  \end{align}
  which is proven in appendix \ref{app:proof}. Applying \eqref{eqn:telescope} to \eqref{eqn:collapsingsum} term by term starting with $r = n$ allows one to step down through the sum and obtain
  \begin{equation}
    x_{\mu \nu} = \frac{y_{\nu-1}^2}{y_{\nu}^2} \frac{\epsilon\left( V_{\mu}, V_{\nu+1}, \dots, V_n; V_{\nu}, V_{\nu+1}, \dots, V_n \right)}{\epsilon\left( V_{\nu+1}, \dots, V_n; V_{\nu+1}, \dots, V_n \right)} = y_{\nu-1}^2 \epsilon\left( V_{\mu}, V_{\nu+1}, \dots, V_n; V_{\nu}, V_{\nu+1}, \dots, V_n \right)
  \end{equation}
  which is exactly \eqref{eqn:axions}.
  
  The dilaton $y_{\mu-1}$ is found through the $(\mu, \mu)$-equation
  \begin{align}
    \begin{aligned}
      \left( NA^2 N^{\T} \right)_{\mu \mu} ={}& \sum_{r=1}^n \sum_{s=1}^n N_{\mu r} \left( A^2 \right)_{rs} \left( N^{\T} \right)_{s \mu} = \sum_{r=1}^n \sum_{s=1}^n N_{\mu r} \frac{y_r^2}{y_{r-1}^2} \delta_{rs} N_{\mu s} = \\
      ={} & \sum_{r=1}^n N_{\mu r} \frac{y_r^2}{y_{r-1}^2} N_{\mu r} = \sum_{r=\mu}^n N_{\mu r} \frac{y_r^2}{y_{r-1}^2} N_{\mu r} \\
      ={} & \frac{y_\mu^2}{y_{\mu-1}^2} + \sum_{r=\mu+1}^n x_{\mu r}^2 \frac{y_r^2}{y_{r-1}^2}
      .
    \end{aligned}
  \end{align}
  Solving for $y_{\mu-1}^{-2}$ gives
  \begin{align}
    y_{\mu-1}^{-2} = \frac{1}{y_{\mu}^2} \left( \epsilon\left( V_\mu; V_\mu \right) - \sum_{r=\mu+1}^n x_{\mu r}^2 \frac{y_r^2}{y_{r-1}^2} \right)
    .
  \end{align}
  We assume again that all ``lower'' variables are given of the form of \cref{eqn:realdilatons,eqn:axions}. The sum then telescopes through \cref{eqn:telescope} with $\nu = \mu$ in precisely the same way as above. The result is
  \begin{align}
    \begin{aligned}
      y_{\mu-1}^{-2} {}&= \frac{1}{y_{\mu}^2} \frac{\epsilon\left( V_\mu, V_{\mu+1}, \dots, V_n; V_\mu, V_{\mu+1}, \dots, V_n \right)}{\epsilon\left( V_{\mu+1}, \dots, V_n; V_{\mu+1}, \dots, V_n \right)} = \\
      &{}= \epsilon\left( V_\mu, V_{\mu+1}, \dots, V_n; V_\mu, V_{\mu+1}, \dots, V_n \right)
    \end{aligned}
  \end{align}
  which is exactly \cref{eqn:realdilatons}.
\end{proof}
\begin{remark}
  The matrix $K$ is given by solving equation \eqref{eqn:realiwasawa} for $K$.
\end{remark}

\section{Representation theoretic viewpoint}
Here we outline a quicker derivation for the norms of the dilatons relying on representation theory. I would like to thank Stephen D.\ Miller and Solomon Friedberg for helpful comments.

Throughout this section we assert that $k \in \{1, \dots, n-1\}$ and use the notation $\rats_\infty \equiv \reals$. We start with a definition.
\begin{defn}[Generalized Pl\"ucker coordinates]
  For a matrix $M \in \SL(n, F)$ over a field $F$, the $k$\textsuperscript{th} generalized Pl\"ucker coordinate $p_k(g) \in F^{\binom{n}{k}}$ is the vector consisting of all anti-leading minors of order $k$ of $M$.
\end{defn}
See \cite{gel1987combinatorial,fomin2000recognizing,friedberg1987poincare,goldfeld2006automorphic} for general definitions and some interesting properties.

Let $\omega_k$ denote the fundamental weights of $\SL(n)$ and $(\rho_k, V_k)$ denote the associated fundamental representation $\rho_k$ with highest weight $\omega_k$ and representation space $V_k$. These are algebraic representation and hence make sense both for the real and $p$-adic cases. Furthermore, let $v_k \in V_k$ denote the highest weight vector. One then has \cite{fomin2000recognizing}
\begin{equation}
  \rho_k\left( g^{-1} \right) v_k = p_{n-k}(g)
  .
\end{equation}

Let $||\cdot||_p$ for $p \leq \infty$ and arbitrary $N \in \nats$ denote the vector norm on $\rats_p^{N}$ which is invariant under the maximal compact subgroup of $\SL(N, \rats_p)$.  More explicitly, in the $p$-adic case $p < \infty$ we have
\begin{equation}
  ||k v||_p = ||v||_p \equiv \max_{i=1}^n |v_i|_p, \quad p<\infty
\end{equation}
for a matrix $k \in \SL(n, \ints_p)$ and $v \in \rats_p^n$ and in the real case $p = \infty$
\begin{equation}
  ||k v||_\infty = ||v||_\infty \equiv \sqrt{\sum_{i=1}^n v^2_i}
\end{equation}
for a matrix $k \in \SO(n)$ and $v \in \reals^n$.
Furthermore for a group element $g \in \SL(n, \rats_p)$ with $p\leq \infty$ written in Iwasawa-form $g=nak$, one may show that
\begin{equation}
  ||\rho_k(g^{-1}) v_k||_p = ||\rho_k(k^{-1})\rho_k(a^{-1})\underbrace{\rho_k(n^{-1}) v_k}_{v_k}||_p = ||\rho_k(a^{-1}) v_k||_p = |y_k|_p^{-1}, \quad p\leq\infty
  .
\end{equation}
We therefore have
\begin{equation}
  |y_k|_p = ||p_{n-k}(g)||_{p}^{-1} , \quad p\leq\infty
\end{equation}
which is precisely \eqref{eqn:min} for the $p$-adic case $p < \infty$. That this formula agrees with \eqref{eqn:realdilatons} in the real case $p=\infty$ is an indirect proof of the following lemma.
\begin{lemma}
  Given a matrix $M \in \GL(n, F)$ over a field $F$, denote the row $n$-vectors in $M$ by $V_i$ for $i\in\{1, \dots, n\}$. One then has the identity
  \begin{align}
    \epsilon\left( V_{k+1}, \dots, V_{n}; V_{k+1}, \dots, V_{n} \right) = \sum_{\sigma \in \Theta_{n}^{n-k}}
  \left(
  M\left(
  \begin{smallmatrix}
    k+1 & \dots & n \\
    \sigma(1) & \dots & \sigma(n-k) \\
  \end{smallmatrix}
  \right)
  \right)^{2}
  \end{align}
\end{lemma}
\begin{proof}
  Indirect by the results above, or by direct calculation using induction and similar techniques as in appendix \ref{app:proof}.
\end{proof}

\section*{Acknowledgements}
I would like to thank Axel Kleinschmidt for helpful discussions over the course of this work. I am supported by the Erasmus Mundus Joint Doctorate Program by grant number 2013-1471 from the EACEA of the European Commission.

\appendix
\section{\label{app:proof}Proving \eqref{eqn:telescope}}
After putting in explicit expressions for the axions and dilatons, proving \cref{eqn:telescope} is equivalent to proving
\begin{align}
  \begin{aligned}
    X\equiv {} &\epsilon\left( V_{\mu}, V_{r}, \dots, V_n; V_{\nu}, V_{r}, \dots, V_n \right) \epsilon\left( V_{r+1}, \dots, V_n; V_{r+1}, \dots, V_n \right) = \\
    ={} &\epsilon\left( V_{\mu}, V_{r+1}, \dots, V_n; V_{\nu}, V_{r+1}, \dots, V_n \right) \epsilon\left( V_{r}, \dots, V_n; V_{r}, \dots, V_n \right) \\
    -{} & \epsilon\left( V_{\mu}, V_{r+1}, \dots, V_n; V_{r}, \dots, V_n \right) \epsilon\left( V_{\nu}, V_{r+1}, \dots, V_n; V_{r}, \dots, V_n \right) \equiv Y - Z
  \end{aligned}
\end{align}
To this end, one may use Laplace expansion for the generalized Kronecker delta
\begin{align}
  \delta{}_{A}^{I}{}_{a_1}^{i_1}{}_{\dash}^{\dash}{}_{a_m}^{i_m} = \delta_{A}^{I} \delta{}_{a_1}^{i_1}{}_{\dash}^{\dash}{}_{a_m}^{i_m} - \sum_{k=1}^m \delta_A^{i_k} \delta{}_{a_1}^{i_1}{}_{\dash}^{\dash}{}_{a_{k-1}}^{i_{k-1}}{}_{a_k}^{I}{}_{a_{k+1}}^{i_{k+1}}{}_{\dash}^{\dash}{}_{a_m}^{i_m} = \delta_{A}^{I} \delta{}_{a_1}^{i_1}{}_{\dash}^{\dash}{}_{a_m}^{i_m} - \sum_{k=1}^m \delta_{a_k}^{I} \delta{}_{a_1}^{i_1}{}_{\dash}^{\dash}{}_{a_{k-1}}^{i_{k-1}}{}_{A}^{i_{k}}{}_{a_{k+1}}^{i_{k+1}}{}_{\dash}^{\dash}{}_{a_m}^{i_m}
  .
\end{align}
In what follows, we will not write out the $V$'s. The expressions are assumed to be fully contracted with vectors corresponding to the indices. For example, an index $a_{r+1}$ is assumed to be contracted with $\left( V_{r+1} \right)^{a_{r+1}}$ and in particular the index $A$ is contracted with $\left( V_{\mu} \right)^A$ and $I$ is contracted with $\left( V_{\nu} \right)_{I}$. We get
\begin{align}
  \begin{aligned}
    X&{}= \delta{}_{A}^{I}{}_{a_{r}}^{i_{r}}{}_{\dash}^{\dash}{}_{a_{n}}^{i_{n}} \delta{}_{b_{r+1}}^{j_{r+1}}{}_{\dash}^{\dash}{}_{b_{n}}^{j_{n}} = \\
    &{}= \left( \delta_A^I \delta{}_{a_{r}}^{i_{r}}{}_{\dash}^{\dash}{}_{a_{n}}^{i_{n}} - \sum_{k=r}^{n} \delta_A^{i_k} \delta{}_{a_{r}}^{i_{r}}{}_{\dash}^{\dash}{}_{a_{k-1}}^{i_{k-1}}{}_{a_{k}}^{I}{}_{a_{k+1}}^{i_{k+1}}{}_{\dash}^{\dash}{}_{a_{n}}^{i_{n}} \right) \delta{}_{b_{r+1}}^{j_{r+1}}{}_{\dash}^{\dash}{}_{b_{n}}^{j_{n}} = \\
    &{}= \delta{}_{a_{r}}^{i_{r}}{}_{\dash}^{\dash}{}_{a_{n}}^{i_{n}}\left( \delta{}_{A}^{I}{}_{b_{r+1}}^{j_{r+1}}{}_{\dash}^{\dash}{}_{b_{n}}^{j_{n}} + \sum_{k=r+1}^{n} \delta_{A}^{j_k} \delta{}_{b_{r+1}}^{j_{r+1}}{}_{\dash}^{\dash}{}_{b_{k-1}}^{j_{k-1}}{}_{b_{k}}^{I}{}_{b_{k+1}}^{j_{k+1}}{}_{\dash}^{\dash}{}_{b_n}^{j_n}  \right) \\
    &{}- \sum_{k=r}^{n} \delta{}_{a_{r}}^{i_{r}}{}_{\dash}^{\dash}{}_{a_{k-1}}^{i_{k-1}}{}_{a_{k}}^{I}{}_{a_{k+1}}^{i_{k+1}}{}_{\dash}^{\dash}{}_{a_{n}}^{i_{n}} \left( \delta{}_{A}^{i_k}{}_{b_{r+1}}^{j_{r+1}}{}_{\dash}^{\dash}{}_{b_{n}}^{j_{n}} + \sum_{l=r+1}^{n} \delta_{A}^{j_l} \delta{}_{b_{r+1}}^{j_{r+1}}{}_{\dash}^{\dash}{}_{b_{l-1}}^{j_{l-1}}{}_{b_{l}}^{i_k}{}_{b_{l+1}}^{j_{l+1}}{}_{\dash}^{\dash}{}_{b_{n}}^{j_{n}} \right) = \\
    &{}= Y - Z + \delta{}_{a_{r}}^{i_{r}}{}_{\dash}^{\dash}{}_{a_{n}}^{i_{n}} \sum_{k=r+1}^{n} \delta_{A}^{j_k} \delta{}_{b_{r+1}}^{j_{r+1}}{}_{\dash}^{\dash}{}_{b_{k-1}}^{j_{k-1}}{}_{b_{k}}^{I}{}_{b_{k+1}}^{j_{k+1}}{}_{\dash}^{\dash}{}_{b_n}^{j_n} \\
    &{}- \delta{}_{a_{r}}^{I}{}_{a_{r+1}}^{i_{r+1}}{}_{\dash}^{\dash}{}_{a_{n}}^{i_{n}} \sum_{l=r+1}^{n} \delta_{A}^{j_l} \delta{}_{b_{r+1}}^{j_{r+1}}{}_{\dash}^{\dash}{}_{b_{l-1}}^{j_{l-1}}{}_{b_{l}}^{i_r}{}_{b_{l+1}}^{j_{l+1}}{}_{\dash}^{\dash}{}_{b_{n}}^{j_{n}} \\
    &{}- \sum_{k=r+1}^{n} \delta{}_{a_{r}}^{i_{r}}{}_{\dash}^{\dash}{}_{a_{k-1}}^{i_{k-1}}{}_{a_{k}}^{I}{}_{a_{k+1}}^{i_{k+1}}{}_{\dash}^{\dash}{}_{a_{n}}^{i_{n}} \delta_{A}^{j_k} \delta{}_{b_{r+1}}^{j_{r+1}}{}_{\dash}^{\dash}{}_{b_{k-1}}^{j_{k-1}}{}_{b_{k}}^{i_k}{}_{b_{k+1}}^{j_{k+1}}{}_{\dash}^{\dash}{}_{b_{n}}^{j_{n}} = \\
    &{}
    \begin{aligned}
      = Y - Z + \sum_{k=r+1}^n \delta_A^{j_k} \Big(  &\delta{}_{a_r}^{i_r}{}_{\dash}^{\dash}{}_{a_n}^{i_n} && \delta{}_{b_k}^{I}{}_{b_{r+1}}^{j_{r+1}}{}_{\dash}^{\dash}{}_{b_{k-1}}^{j_{k-1}}{}_{b_{k+1}}^{j_{k+1}}{}_{\dash}^{\dash}{}_{b_n}^{j_n} \\
      -{}& \delta{}_{b_{r+1}}^{j_{r+1}}{}_{\dash}^{\dash}{}_{b_{k-1}}^{j_{k-1}}{}_{b_k}^{i_r}{}_{b_{k+1}}^{j_{k+1}}{}_{\dash}^{\dash}{}_{b_n}^{j_n} && \delta{}_{a_r}^{I}{}_{a_{r+1}}^{i_{r+1}}{}_{\dash}^{\dash}{}_{a_n}^{i_{n}} \\
      -{}& \delta{}_{b_{r+1}}^{i_{r+1}}{}_{\dash}^{\dash}{}_{b_{n}}^{i_{n}} && \delta{}_{a_{k}}^{I}{}_{a_{r}}^{j_{r}}{}_{\dash}^{\dash}{}_{a_{k-1}}^{j_{k-1}}{}_{a_{k+1}}^{j_{k+1}}{}_{\dash}^{\dash}{}_{a_{n}}^{j_{n}}  \Big) \equiv
    \end{aligned} \\
    &{}\equiv Y - Z + \sum_{k=r+1}^{n} \delta_{A}^{j_k} \left( R_1 - R_2 - R_3 \right)
    .
  \end{aligned}
\end{align}
The job is now to show that the remainder vanishes. We expand $R_1$
\begin{align}
  \begin{aligned}
    R_1 &{}= \delta{}_{a_{r}}^{i_{r}}{}_{\dash}^{\dash}{}_{a_{n}}^{i_{n}} \delta{}_{b_{k}}^{I}{}_{b_{r+1}}^{j_{r+1}}{}_{\dash}^{\dash}{}_{b_{k-1}}^{j_{k-1}}{}_{b_{k+1}}^{j_{k+1}}{}_{\dash}^{\dash}{}_{b_{n}}^{j_n} = \\
    &{}= \left( \delta_{a_{r}}^{i_{r}} \delta{}_{a_{r+1}}^{i_{r+1}}{}_{\dash}^{\dash}{}_{a_{n}}^{i_{n}} - \sum_{l=r+1}^n \delta_{a_{r}}^{i_{l}} \delta{}_{a_{r+1}}^{i_{r+1}}{}_{\dash}^{\dash}{}_{a_{l-1}}^{i_{l-1}}{}_{a_{l}}^{i_{r}}{}_{a_{l+1}}^{i_{l+1}}{}_{\dash}^{\dash}{}_{a_{n}}^{i_{n}} \right) \delta{}_{b_{k}}^{I}{}_{b_{r+1}}^{j_{r+1}}{}_{\dash}^{\dash}{}_{b_{k-1}}^{j_{k-1}}{}_{b_{k+1}}^{j_{k+1}}{}_{\dash}^{\dash}{}_{b_{n}}^{j_n} = \\
    &{}= \delta{}_{a_{r+1}}^{i_{r+1}}{}_{\dash}^{\dash}{}_{a_{n}}^{i_{n}} \Bigg( \underbrace{\delta{}_{a_{r}}^{i_{r}}{}_{b_{k}}^{I}{}_{b_{r+1}}^{j_{r+1}}{}_{\dash}^{\dash}{}_{b_{k-1}}^{j_{k-1}}{}_{b_{k+1}}^{j_{k+1}}{}_{\dash}^{\dash}{}_{b_{n}}^{j_n}}_{R_3} + \underbrace{\delta_{b_k}^{i_r} \delta{}_{a_{r}}^{I}{}_{b_{r+1}}^{j_{r+1}}{}_{\dash}^{\dash}{}_{b_{k-1}}^{j_{k-1}}{}_{b_{k+1}}^{j_{k+1}}{}_{\dash}^{\dash}{}_{b_{n}}^{j_n}}_{\text{I}} + \\
    &{}+ \underbrace{\sum_{m=r+1}^{k-1} \delta_{b_m}^{i_r} \delta{}_{a_{r}}^{I}{}_{b_{r+1}}^{j_{r+1}}{}_{\dash}^{\dash}{}_{b_{m-1}}^{j_{m-1}}{}_{a_r}^{j_{m}}{}_{b_{m+1}}^{j_{m+1}}{}_{\dash}^{\dash}{}_{b_{k-1}}^{j_{k-1}}{}_{b_{k+1}}^{j_{k+1}}{}_{\dash}^{\dash}{}_{b_{n}}^{j_n}}_{\text{II}} + \underbrace{\sum_{m=r+1}^{k-1} \delta_{b_m}^{i_r} \delta{}_{a_{r}}^{I}{}_{b_{r+1}}^{j_{r+1}}{}_{\dash}^{\dash}{}_{b_{k-1}}^{j_{k-1}}{}_{b_{k+1}}^{j_{k+1}}{}_{\dash}^{\dash}{}_{b_{m-1}}^{j_{m-1}}{}_{a_r}^{j_{m}}{}_{b_{m+1}}^{j_{j+1}}{}_{\dash}^{\dash}{}_{b_{n}}^{j_n}}_{\text{III}}  \Bigg) \\
    &{}- \sum_{l=r+1}^n \delta{}_{a_{r+1}}^{i_{r+1}}{}_{\dash}^{\dash}{}_{a_{l-1}}^{i_{l-1}}{}_{a_{l}}^{i_{r}}{}_{a_{l+1}}^{i_{l+1}}{}_{\dash}^{\dash}{}_{a_{n}}^{i_{n}} \Bigg(  \underbrace{\delta{}_{a_{r}}^{i_{l}}{}_{b_{k}}^{I}{}_{b_{r+1}}^{j_{r+1}}{}_{\dash}^{\dash}{}_{b_{k-1}}^{j_{k-1}}{}_{b_{k+1}}^{j_{k+1}}{}_{\dash}^{\dash}{}_{b_{n}}^{j_n}}_{R_2} + \underbrace{\delta_{b_{k}}^{i_{l}} \delta{}_{a_{r}}^{}{}_{b_{r+1}}^{I}{}_{\dash}^{j_{r+1}}{}_{b_{k-1}}^{\dash}{}_{b_{k+1}}^{j_{k-1}}{}_{\dash}^{j_{k+1}}{}_{b_{n}}^{\dash j_n}}_{\text{I}} + \\
    &{}+ \underbrace{\sum_{m=r+1}^{k-1} \delta_{b_m}^{i_l} \delta{}_{b_{k}}^{I}{}_{b_{r+1}}^{j_{r+1}}{}_{\dash}^{\dash}{}_{b_{m-1}}^{j_{m-1}}{}_{a_{r}}^{j_{m}}{}_{b_{m+1}}^{j_{m+1}}{}_{\dash}^{\dash}{}_{b_{k-1}}^{j_{k-1}}{}_{b_{k+1}}^{j_{k+1}}{}_{\dash}^{\dash}{}_{b_{n}}^{j_n}}_{\text{II}} + \underbrace{\sum_{m=k+1}^{n} \delta_{b_m}^{i_l} \delta{}_{b_{k}}^{I}{}_{b_{r+1}}^{j_{r+1}}{}_{\dash}^{\dash}{}_{b_{k-1}}^{j_{k-1}}{}_{b_{k+1}}^{j_{k+1}}{}_{\dash}^{\dash}{}_{b_{m-1}}^{j_{m-1}}{}_{a_{r}}^{j_{m}}{}_{b_{m+1}}^{j_{m+1}}{}_{\dash}^{\dash}{}_{b_{n}}^{j_n}}_{\text{III}} \Bigg)
    .
  \end{aligned}
\end{align}
The terms marked I, II and III cancel out by writing
\begin{align}
  \delta_{b_k}^{i_r}\delta{}_{a_{r+1}}^{i_{r+1}}{}_{\dash}^{\dash}{}_{a_{n}}^{i_{n}} = \delta{}_{b_k}^{i_r}{}_{a_{r+1}}^{i_{r+1}}{}_{\dash}^{\dash}{}_{a_{n}}^{i_{n}} + \sum_{l=r+1}^{n} \delta_{b_k}^{i_l}\delta{}_{a_{r+1}}^{i_{r+1}}{}_{\dash}^{\dash}{}_{a_{l-1}}^{i_{l-1}}{}_{a_{l}}^{i_r}{}_{a_{l+1}}^{i_{l+1}}{}_{\dash}^{\dash}{}_{a_n}^{i_n} = \sum_{l=r+1}^{n} \delta_{b_k}^{i_l}\delta{}_{a_{r+1}}^{i_{r+1}}{}_{\dash}^{\dash}{}_{a_{l-1}}^{i_{l-1}}{}_{a_{l}}^{i_r}{}_{a_{l+1}}^{i_{l+1}}{}_{\dash}^{\dash}{}_{a_n}^{i_n}
\end{align}
and similarly for terms II and III with $\delta_{b_m}^{i_r}\delta{}_{a_{r+1}}^{i_{r+1}}{}_{\dash}^{\dash}{}_{a_{n}}^{i_{n}}$.

\bibliography{bib/references}{}
\bibliographystyle{utphys}

\end{document}